\documentclass[10pt]{article}

\usepackage[all]{xy}
\usepackage{amsmath}
\usepackage{amsfonts}
\usepackage{amssymb}
\usepackage{amsthm}
\usepackage{bbm}
\usepackage{epsfig}

\newtheorem{theorem}{Theorem}[section]

\newtheorem{proposition}[theorem]{Proposition}
\newtheorem{corollary}[theorem]{Corollary}

\newtheorem{claim}[theorem]{Claim}
\newtheorem{definition}[theorem]{Definition}
\newtheorem{algorithm}[theorem]{Algorithm}

\newtheorem{fact}[theorem]{Fact}

\newenvironment{prf}[1]{\trivlist 
\item[\hskip \labelsep{\it 
#1.\hspace*{.3em}}]}{~\hspace{\fill}~$\square$\endtrivlist} 

\newcommand{\xg}{\mathbb{G}}

\newcommand{\xp}{\mathbb{P}}
\newcommand{\xq}{\mathbb{Q}}
\newcommand{\xf}{\mathbb{F}}

\newcommand{\xz}{\mathbb{Z}}

\newcommand{\spec}[1]{\mathrm{Spec}(#1)}

\title{Explicit Frobenius lifts on elliptic curves}

\author{Robert Carls}

\begin{document}

\maketitle

{\small 
\begin{center}
\begin{tabular}{ll}
Robert Carls \\
\texttt{robert.carls@uni-ulm.de} \\
Universit\"{a}t Ulm \\
Institut f\"{u}r Reine Mathematik \\
D-89069 Ulm, Germany
\end{tabular}
\end{center}
}

\begin{abstract}
\noindent
In this article we give explicit formulae for a lift
of the relative Frobenius morphism between elliptic curves
and show how one can compute
this lift in the case of ordinary reduction in odd
characteristic. Our theory can also be used in the
case of supersingular reduction.
\newline
By means of the explicit formulae that describe a Frobenius lift,
we are able to generalize Mestre's $2$-adic arithmetic
geometric mean (AGM) sequence of elliptic curves
to odd characteristic, and prove its convergence.
As an application, we give an efficient point counting algorithm for ordinary
elliptic curves which is based
on the generalized AGM sequence.
\end{abstract}

\noindent
{\small
\begin{center}
\begin{tabular}{ll}
KEYWORDS: & Frobenius Lift, Point Counting, Arithmetic Geometric Mean
\end{tabular}
\end{center}
}

\section{Introduction}
\label{intro}

In this article we give formulae
which describe a lift of the relative Frobenius morphism
for a given elliptic curve over the $p$-adic numbers,
where $p$ is an odd prime.
These formulae are universal in the sense that
they can be used in the case of ordinary reduction,
and in the supersingular case as well.
\newline\indent
Let us first give an example in the case of ordinary
reduction.
Consider the elliptic curve $E$ over $\xq$
which is given by the equation
\[
y^2=x(x-1)(x-i), \quad i^2=-1.
\]
The curve $E$ has ordinary good reduction at the
prime $3$.
The normalized third division polynomial of $E$ is given by
\[
\psi_3(T)=T^4 + \frac{1}{3}(-4i - 4)T^3 + 2iT^2 + \frac{1}{3}.
\]
We set $K=\xq(i)[T]/(\psi_3)$.
Let $\bar{t}$ denote the class of $T$ in the quotient ring $K$.
Over $K$ the polynomial $\psi_3$ decomposes into
irreducible factors as follows
{\scriptsize
\[
\big( T-\bar{t} \big) \cdot \left(
T^3+(\bar{t}+\frac{1}{3}(-4i-4))T^2
+(\bar{t}^2+\frac{1}{3}(-4i- 4)\bar{t}+2i)T+(\bar{t}^3+ 
\frac{1}{3}(-4i-4)\bar{t}^2+2i\bar{t})
\right).
\]
}
In the following we consider the curve $E$ as defined
over the field $K$.
There exists an isogeny $F:E \rightarrow E^{(3)}$, where the curve
$E^{(3)}$ is given by the
equation
\[
y^2=x \cdot \left( x-\frac{i-\bar{t}}{1-\bar{t}} \right)
\cdot \left( x+\frac{i(1-\bar{t})}{i-\bar{t}} \right),
\]
such that $F$ reduces modulo $3$ to the relative
$3$-Frobenius morphism.
Also, using the formulae of Theorem \ref{froblift}, one can
compute explicit formulae for the isogeny $F$ as a pair
of rational functions. We will not go into the details
of this computation here.
\newline\indent
Now consider the elliptic curve $E$ given by the equation
\[
y^2=x(x-1)(x-2)
\]
over $\xq$.
The curve $E$ has supersingular good reduction at the prime $3$.
The normalized third division polynomial of $E$ is given by
\[
\psi_3(T)=T^4 - 4T^3 + 4T^2 - \frac{4}{3}.
\]
We set $K=\xq[T]/(\psi_3)$.
Let $\bar{t}$ denote the class of $T$ in the quotient $K$.
Over $K$ the polynomial $\psi_3$ decomposes in irreducible
factors as follows
\[
\big( T-\bar{t} \big) \cdot
\big( T-(2-\bar{t}) \big) \cdot
\big( T^2 - 2T + (\bar{t}^2 - 2 \bar{t}) \big).
\]
In the following we consider $E$ to be defined
over $K$. Now assume that we have chosen a
$3$-adic embedding of $K$ such that $\bar{t}$
is a non-unit.
Then the image curve $E^{(3)}$ of a lift of the relative
Frobenius $F:E \rightarrow E^{(3)}$ can be given by
\[
y^2=x \cdot \left( x- \frac{2-\bar{t}}{1-\bar{t}}
\right) \cdot \left( x -\frac{8(1-\bar{t})}{2-\bar{t}} \right).
\]
\indent
Next let us explain the relevance of the results of this article
to the algorithmic application.
Let $A$ be an abelian scheme over a $p$-adic local ring, and let $A$ have ordinary
reduction. It is a classical result that there exists
a canonical lift $F:A \rightarrow A^{(p)}$ of the
relative $p$-Frobenius morphism. By iterating the lifting one obtains
a sequence of abelian schemes
\[
A \stackrel{F}{\rightarrow} A^{(p)} \stackrel{F}{\rightarrow}
A^{(p^2)}
\rightarrow \ldots
\]
Let $A_0$ denote the reduction of $A$, where we assume
$A_0$ to be defined over a finite field $\xf_q$ with $q$ elements
of characteristic $p>0$.
There exists a canonical lift $A^*$ of $A_0$
which is characterized by the property that the reduction
map on endomorphisms is bijective.
It is a fundamental result that one has
\begin{eqnarray}
\label{limess}
\lim_{n \rightarrow \infty} A^{(q^n)} = A^*
\end{eqnarray}
with respect to the $p$-adic topology.
The precise statement and a proof are given in Section \ref{sectconv}.
\newline\indent
In the following let $p=2$.
We restrict our attention to the case where $A$ is an elliptic curve. 
The convergence theorem (\ref{limess}) forms the basis
of Mestre's AGM point counting algorithm (see the private
conversation \cite{me00}). An essential step in this algorithm
is to compute the arithmetic geometric mean (AGM) sequence
\begin{eqnarray}
\label{2adicagm}
(a_n,b_n)=\left( \frac{a_{n-1}+b_{n-1}}{2},\sqrt{a_{n-1} b_{n-1}} \right)
\end{eqnarray}
in a $2$-adic local field with finite residue field $\xf_q$
of characteristic $2$.
It turns out that the sequence (\ref{2adicagm}) describes
the coefficients of a sequence of elliptic curves with
ordinary good reduction
\begin{eqnarray}
\label{eagm}
E_n:y^2=x(x-a_n^2)(x-b_n^2), \quad n \geq 0,
\end{eqnarray}
and Frobenius lifts
\[
E_0 \stackrel{F}{\rightarrow} E_1 \stackrel{F}{\rightarrow}
E_2
\rightarrow \ldots
\]
where by (\ref{limess}) the subsequence $E_{dm}$
with $d=\log_2(q)$ approximates $2$-adically
the canonical lift of the reduction of $E_0$.
For a higher dimensional generalization of Mestre's
algorithm see the informal notes
\cite{me02}. The theoretical background of
the higher dimensional AGM sequence is given in \cite{ca05b}.
\newline\indent
In this article we give an analogue of Mestre's
sequence of elliptic curves (\ref{eagm}) in
odd residue class field characteristic.
Using the explicit formulae for a lift of the
relative Frobenius, we are able to define a
$2$-parameter analogue of Mestre's $2$-adic
AGM sequence.
We apply our results
to the point
counting problem on ordinary elliptic curves over finite fields
of odd characteristic.
We prove that, if $E$ is
an ordinary elliptic curve over a finite field $\xf_q$
with $q$ elements of characteristic $p>2$, then
one can give an algorithm for the computation of the
number of rational points $\#E(\xf_q)$ which has
the time complexity $O \big( p^{2+ \epsilon} \log_p(q)^{3 + \epsilon}
\big)$
for all $\epsilon>0$.
The algorithm that we give is based on the computation
of the $p$-adic analogue of Mestre's $2$-adic
AGM sequence.

\subsection*{Leitfaden}

In Section \ref{sectconv} we give a proof of the convergence theorem
for the $p$-adic scheme theoretic analogue of Mestre's AGM sequence.
In Section \ref{explicit} we give explicit formulae for
a lift of the relative Frobenius in the elliptic curve case.
In Section \ref{alggo} we give an algorithmic application
of these formulae.
In Section \ref{exagm} we give an example of the
generalized $p$-adic AGM sequence.

\section{Canonical Frobenius lift in the ordinary case}
\label{sectconv}

In this section we give some theoretical background
on Frobenius lifts.
Let $R$ be a complete noetherian local ring. We assume
$R$ to have perfect residue class
field $k$ of characteristic $p>0$. By $\mathfrak{m}_R$ we denote
the maximal ideal of $R$.
Let $\pi:A \rightarrow \spec{R}$ be an abelian scheme
which has ordinary reduction.
\begin{proposition}
\label{frobeniusoverr}
There exists an abelian scheme $\pi^{(p)}:A^{(p)} \rightarrow
\spec{R}$ and a commutative diagram of isogenies
\begin{eqnarray*}
\xymatrix{
  A \ar@{->}[r]^F \ar@{->}[d]_{[p]} & A^{(p)} \\
  A \ar@{<-}[ur]_V }
\end{eqnarray*}
such that the reduction
$F_k$ of $F$ fits into the following
diagram commutative
\[
\xymatrix{ A_k \ar@/^/[rrd]^{f_p}
  \ar@/_/[rdd]_{\pi_k} \ar@{->}[rd]|-{F_k} & & \\
  & A_k^{(p)} \ar@{->}[r]^{\mathrm{pr}} \ar@{->}[d]^{\pi^{(p)}_k} & A_k
  \ar@{->}[d]^{\pi_k} \\
  & \spec{k} \ar@{->}[r]^{f_p} & \spec{k} }
\]
where $f_p$ denotes the absolute $p$-Frobenius and
where $\mathrm{pr}:A_k^{(p)} \rightarrow A_k$
is a morphism which makes the square Cartesian.
In other words, the isogeny $F$ is a lift of the relative
$p$-Frobenius morphism.
The isogeny $F$ is uniquely determined by the condition
\[
\mathrm{Ker}(F)=A[p]^{\mathrm{loc}}
\]
where $A[p]^{\mathrm{loc}}$ denotes the connected component
of the zero element in the finite flat group $A[p]$.
\end{proposition}
\begin{prf}{Proof of Proposition \ref{frobeniusoverr}}
In order to prove Proposition \ref{frobeniusoverr} we need the
following classical result.
\begin{claim}
\label{quotients}
Let $A$ be an abelian scheme over a noetherian ring $R$ and
$G$ a finite flat subgroup of $A$. Then the quotient sheaf $A/G$
is representable by an abelian scheme. The quotient map
$A \rightarrow A/G$ is an isogeny.
\end{claim}
\begin{prf}{Proof of Claim \ref{quotients}}
We only sketch the proof of the claim.
First assume that $R$ is integral and normal.
Then $A$ is projective (see \cite[Ch.XI,Th.1.4]{ra70b}).
Projectivity implies that every $G$-orbit lies in some
open affine.
As a consequence the
quotient $A/G$ is representable (compare \cite[$\S$5,Th.1]{ra67}).
The general case can be deduced from the above special case
by the method which is used in the proof of \cite[Ch.I,Th.1.9]{fc90}.
\end{prf}
\noindent
Now let $A$ and $R$ be as in the proposition.
It is a classical result that there exists an exact sequence
of groups
\[
0 \rightarrow A[p]^{\mathrm{loc}} \rightarrow
A[p] \rightarrow A[p]^{\mathrm{et}} \rightarrow 0
\]
where $A[p]^{\mathrm{loc}}$ is the connected component of the
zero element in $A[p]$ and $A[p]^{\mathrm{et}}$ is the
maximal \'{e}tale quotient of $A[p]$
by $A[p]^{\mathrm{loc}}$ (for a proof see \cite[$\S$3.7]{ta97}).
By Claim \ref{quotients} there exists an isogeny
\[
F:A \rightarrow A^{(p)} \stackrel{\mathrm{def}}{=} A/A[p]^{\mathrm{loc}}.
\]
Because of  the relation
$\mbox{Ker}(F)=A[p]^{\mathrm{loc}} \subseteq A[p]$
there exists a commutative diagram
\begin{eqnarray*}
\xymatrix{
  A \ar@{->}[r]^F \ar@{->}[d]_{[p]} & A^{(p)} \\
  A \ar@{<-}[ur]_V }
\end{eqnarray*}
One checks fiberwise that the morphism $V$ is an isogeny.
We claim that
$A_k[p]^{\rm{loc}} = \big( A[p]^{\rm{loc}} \big)_k$.
In the following we prove our claim.
Let $A[p]^{\rm{loc}}=\spec{C}$.
The finite $R$-algebra $C$ is connected, i.e.
it has only the trivial idempotents $0$ and $1$.
The same is true for $C \otimes_R k$, since $R$ is henselian and
thus one can lift the idempotents of $C \otimes_R k$ to $C$
(compare \cite{ra70} Ch. I, $\S$1).
This proves our claim.
\newline\indent
The kernel $K$ of the relative
Frobenius morphism of $A_k$
has no non-zero points over the algebraic closure of $k$
and hence is connected. It follows that
$K \subseteq A_k[p]^{\rm{loc}}$.
The order of $K$ equals $p^g$
where $g$ is the relative dimension of $A$ over $R$.
By comparing ranks we conclude that the equality
$K = A_k[p]^{\rm{loc}}$ holds.
This proves the proposition.
\end{prf}
\noindent
By successively dividing out the connected component
of the $p$-torsion one obtains
a sequence of abelian schemes with bonding morphisms
\[
A \stackrel{F}{\rightarrow} A^{(p)} \stackrel{F}{\rightarrow}
A^{(p^2)}
\rightarrow \ldots
\]
In Section \ref{explicit} we show how this sequence can be made
explicit in the case of elliptic curves.
\begin{theorem}
\label{teichmueller}
Let $B$ be an abelian scheme over $R$ and let $i \geq 1$
be an integer.
If
\[
A \cong B \bmod \mathfrak{m}_R^i
\]
then
\[
A^{(p)} \cong B^{(p)} \bmod \mathfrak{m}_R^{i+1}.
\]
\end{theorem}
\begin{prf}{Proof of Theorem \ref{teichmueller}}
The following proof of the theorem was communicated
to the author by Moonen.
We only sketch his proof.
For a more detailed version of the proof see
\cite[Ch.2]{ca04}.
\newline\indent
In the following we give a proof in the case
where $k$ is algebraically closed.
In the following let $S$ denote a complete noetherian
local ring with algebraically closed residue class field
$k$ of characteristic $p>0$.
Let $W(k)$ denote the Witt vectors with values in $k$.
Note that there exists a canonical morphism $W(k) \rightarrow S$.
Assume that we are given an ordinary abelian variety
$A_0$ over $k$. Let $\mathrm{Def}(A_0)$ denote the functor
which associates to every local ring $S$ as above
the set of isomorphism classes of pairs $(A,\varphi)$,
where $A$ is a formal abelian scheme over $S$
and $\varphi$ is an isomorphism $A_k \stackrel{\sim}{\rightarrow}
A_0$.
\begin{fact}
\label{ldef}
The functor $\mathrm{Def}(A_0)$ is representable
by a split formal torus of relative dimension $g^2$ over $W(k)$,
where $g= \mathrm{dim}(A_0)$. The canonical lift of $A_0$
corresponds to the unit element of the formal torus.
\end{fact}
\noindent
Recall that a split formal torus of relative
dimension $1$ over $W(k)$, denoted by $\hat{\xg}_{m,W(k)}$,
is defined as
the completion of the multiplicative group $\xg_{m,W(k)}$
at its unit section.
In higher dimension, we call a formal group over $W(k)$ a formal torus
if it is isomorphic to a
product of copies of the formal group $\hat{\xg}_{m,W(k)}$.
\newline\indent
Now let $F_0:A_0 \rightarrow A_0^{(p)}$ denote the
relative $p$-Frobenius morphism.
By Proposition \ref{frobeniusoverr} there exists
a canonical induced morphism
\begin{eqnarray}
\label{defmap}
\mathrm{Def}(A_0) \rightarrow \mathrm{Def}\big( A_0^{(p)} \big)
\end{eqnarray}
which maps $(A,\varphi)$ to $\big( A^{(p)},\varphi^{(p)} \big)$.
Assume that we have chosen an isomorphism
$\mathrm{Def}(A_0) \stackrel{\sim}{\rightarrow}
\hat{\xg}_{m,W(k)}^{g^2}$.
\begin{fact}
\label{oper}
There exists precisely one structure of a formal torus
on $\mathrm{Def} \big( A_0^{(p)} \big)$,
i.e. an isomorphism $\mathrm{Def} \big( A_0^{(p)} \big)
\stackrel{\sim}{\rightarrow}
\hat{\xg}_{m,W(k)}^{g^2}$,
such that the morphism
{\rm (\ref{defmap})} is a homomorphism of formal groups.
With respect to this choice of a group structure the morphism
{\rm (\ref{defmap})} is given by the $p$-th powering morphism.
\end{fact}
\noindent
The group functor $\hat{\xg}_{m,W(k)}$
associates to every local ring $S$ as above
the multiplicative subgroup $1 + \mathfrak{m}$ of $S^*$ where
$\mathfrak{m}$ denotes the maximal ideal of $S$. 
The theorem now follows from the following computation
\[
(1+ m)^p= \sum_{i=0}^p {p \choose i} \cdot m^{p-i}
\in 1+m^p+pmR \subseteq
1+\mathfrak{m}^{i+1}, \quad \forall m \in \mathfrak{m^i}.
\]
\end{prf}
\noindent
Let $A^*$ denote the canonical lift of $A_k$.
We note that if $k$ is a finite field with $q=\#k$ elements
then for all $n \geq 0$ there exists an isomorphism
$A_k \cong A_k^{(q^n)}$.
\begin{corollary}
\label{conv}
Let $q=\# k < \infty$.
One has
\[
\lim_{n \rightarrow \infty} A^{(q^n)} = A^*,
\]
which means that
\[
(\forall n \geq N)
\quad  A^{(q^n)} \cong A^* \bmod \mathfrak{m}^{nd+1}
\]
where $d=\log_p(q)$.
\end{corollary}
\begin{proof}
The claim is an immediate consequence of Theorem \ref{teichmueller}. 
We note that one has
\[
\big( A^* \big)^{(q)} \cong A^*.
\]
\end{proof}

\section{Explicit formulae for Frobenius lifts}
\label{explicit}

In this section we show how the Frobenius lift
whose existence is proven in Proposition \ref{frobeniusoverr}
can be made explicit.
The special case of even residue class field
characteristic is discussed in Section \ref{agm}.
\newline\indent
Let $R$ denote a complete discrete valuation ring with perfect residue field
$k$ of characteristic $p>2$ and $K$ its field of fractions.
Let $\pi$ be a uniformizer of $R$, let $v:R \rightarrow \xz$
be an exponential valuation such that $v(\pi)=1$ and
let $\bar{K}$ be an algebraic closure of $K$.
We assume $K$ to have characteristic $0$.
\newline\indent
Consider an elliptic curve $E$ over $K$ which has good reduction.
We assume that $\#E[2](K)=4$.
The curve $E$ admits a model
\begin{eqnarray}
\label{weiermin}
y^2=x(x-a)(x-b)
\end{eqnarray}
where $a,b \in R^*$ and $a \not\equiv b \bmod \pi$.
Let the reduction $\bar{E}$ of $E$
be given by
\[
y^2=x(x-\bar{a})(x-\bar{b}),
\]
where the coefficients
\[
\bar{a} \equiv a \bmod \pi \quad \mbox{and} \quad \bar{b} \equiv b \bmod \pi
\]
are in $k$.
Let $\bar{E}^{(p)}$ over $k$ be the elliptic curve with equation
\[
y^2=x(x-\bar{a}^p)(x-\bar{b}^p).
\]
The isogeny $\bar{F}:\bar{E} \rightarrow \bar{E}^{(p)}$
over $k$, which is defined by $(x,y) \mapsto (x^p,y^p)$ is called the
relative Frobenius morphism.
In the following we will discuss necessary conditions for the
existence of a lift
of the relative Frobenius morphism.
\newline\indent
Assume that we are given a subgroup $G \leq E[p](\bar{K})$ of order
$p$, which is defined over $K$, i.e. $\sigma(G) = G$ for all
$\sigma \in \mathrm{Gal}(\bar{K}/K)$.
Let $S \subseteq G$ such that $S \cap -S = \emptyset$
and $G=S \cup -S \cup \{ 0_E \}$, where $0_E$ denotes the zero section
of $E$.
We set
\[
P_1=(0,0), \quad P_2=(a,0) \quad \mbox{and} \quad P_3=(b,0). 
\]
Let $x(Q)$ denote the $x$-coordinate of a point $0_E \not= Q \in E(\bar{K})$.
We define
\[
h(x)=\prod_{Q \in S} \left( \frac{x}{x(Q)}-1 \right)
\]
and
\[
g_i(x)= \prod_{Q \in S} \big( x-x(Q+P_i) \big), \quad i=1,2,3. 
\]
Note that $x(Q) \not=0$ for $Q \in S$, since $p>2$.
Since $G$ is defined over $K$, the polynomials $h(x)$ and
$g_i(x)$ are elements of $K[x]$, where $i=1,2,3$.
\begin{theorem}
\label{froblift}
Assume that we are given a subgroup $G \subseteq E[p](\bar{K})$
with $\#G=p$, which is defined over $K$.
Suppose $G$ is contained in the kernel of reduction.
Let $E^{(p)}$ be defined by
\[
y^2=x \big( x-a^{(p)} \big) \big( x-b^{(p)} \big) ,
\]
where
\[
a^{(p)}=a^p \cdot \left( \frac{h(b)}{h(a)}
\right)^2
\quad \mbox{and} \quad
b^{(p)}=b^p \cdot \left( \frac{h(a)}{h(b)}
\right)^2.
\]
Then $E^{(p)}$ is a non-singular elliptic curve, and
there exists an isogeny
\[
F:E \rightarrow E^{(p)}
\]
given by
\[
(x,y) \mapsto \left( \frac{x \cdot g_1(x)^2}{h(x)^2},
\frac{\prod_{i=1}^3 g_i(x)}{h(x)^3} \cdot y \right)
\]
which has kernel $G$.
We have $a^{(p)},b^{(p)} \in R$ and $h(x), g_i(x) \in R[x]$,
where $i=1,2,3$.
The curve $E^{(p)}$ reduces to $\bar{E}^{(p)}$
and the isogeny $F$ lifts the relative Frobenius $\bar{F}$.
Also we have
\[
F^* \left( \frac{dx}{y} \right) = \mathrm{lead}(h) \cdot \frac{dx}{y},
\]
where $\mathrm{lead}(h)$ denotes the leading coefficient of
the polynomial $h(x)$.
\end{theorem}
\begin{proof}
Abstract theory (see \cite{ra67}) guarantees the existence of a
quotient $E^{(p)}$ of $E$ by $G$.
We will construct suitable coordinate functions $\tilde{x}$
and $\tilde{y}$ on $E^{(p)}$
using the coordinates $x$ and $y$ on $E$.
Consider the functions $\tilde{x},\tilde{y}:E \rightarrow \xp^1_K$ given by
\[
\tilde{x}(x,y)=\frac{x \cdot g_1(x)^2}{h(x)^2}
\]
and
\[
\tilde{y}(x,y)=\frac{ \prod_{i=1}^3 g_i(x)}{h(x)^3} \cdot y.
\]
The function $\tilde{x}$ has divisor
\begin{eqnarray}
\label{divx}
2 \cdot \sum_{Q \in G} \big( Q+P_1 \big) -
2 \cdot \sum_{Q \in G} \big( Q \big)
\end{eqnarray}
and the function $\tilde{y}$ has divisor
\begin{eqnarray}
\label{divy}
\sum_{i=1}^3 \sum_{Q \in G} \big( Q+P_i \big)
-3 \cdot \sum_{Q \in G} \big( Q \big).
\end{eqnarray}
We claim that
the functions $\tilde{x}$ and $\tilde{y}$ are $G$-invariant, i.e.
invariant under the composition
with translations given by points of $G$.
There exist $G$-invariant functions on $E$
having the divisors (\ref{divx}) and (\ref{divy}).
The latter is due to the fact that
by abstract theory the quotient $E^{(p)}$ exists and we
can pull back suitable coordinate functions. They
differ from $\tilde{x}$ resp. $\tilde{y}$
by a constant. This implies the claim.
\newline\indent
We claim that the functions $\tilde{x}$ and $\tilde{y}$ defined above
satisfy the equation
\begin{eqnarray}
\label{neweq}
(\tilde{y})^2=\tilde{x} \cdot \big( \tilde{x}-a^{(p)} \big)
\cdot \big( \tilde{x}-b^{(p)} \big) .
\end{eqnarray}
One computes
\[
a \cdot g_1(a)^2=a \cdot \prod_{Q \in S} \big( a-x(Q+P_1) \big)^2
= a \cdot \prod_{Q \in S} a^2 \cdot \left( 1-\frac{b}{x(Q)} \right)^2
=a^p \cdot h(b)^2
\]
and hence by definition
\[
\tilde{x}(P_2)=a^{(p)}.
\]
Similarly one gets
\[
\tilde{x}(P_3)=b^{(p)}.  
\]
It follows that the divisors of
$(\tilde{y})^2$ and
\[
\tilde{x} \cdot \big( \tilde{x}-a^{(p)} \big)
\cdot \big( \tilde{x}-b^{(p)} \big)
\]
are equal.
Hence these two functions differ by a constant.
We determine the constant by looking at expansions in $z=-\frac{x}{y}$.
One has
\[
x(z)=\frac{1}{z^2}+ \ldots \quad \mbox{and} \quad
y(z)=-\frac{1}{z^3} + \ldots
\]
We set $l=\mathrm{lead}(h)$.
Then
\[
h \big( x(z) \big)=\frac{l}{z^{p-1}} + \ldots
\]
Since the $g_i(x)$, where $i=1,2,3$, are monic we get
\[
\tilde{x} \big( x(z),y(z) \big)=\frac{\frac{1}{z^{2p}}+ \ldots}
{\frac{l^2}{z^{2p-2}}+\ldots}
\quad \mbox{and} \quad
\tilde{y} \big( x(z),y(z) \big)=
\frac{-\frac{1}{z^{3p}}+ \ldots}
{\frac{l^3}{z^{3p-3}}+\ldots}
\]
Hence the above mentioned constant equals $1$ and
equality (\ref{neweq}) holds. This proves our claim.
\newline\indent
Next we prove that the curve $E^{(p)}$ given by
equation (\ref{neweq})
and the morphism $F:E \rightarrow E^{(p)}$ given by
\[
(x,y) \mapsto \big( \tilde{x}(x,y),\tilde{y}(x,y) \big)
\]
are defined over $R$.
Let $Q \in S$.
Using the addition formulae (see \cite{si86} Ch. III, $\S$2) we
compute
\begin{eqnarray}
\label{coord1}
& & x(Q+P_1)=\frac{ab}{x(Q)}, \\
\nonumber & & x(Q+P_2)=\frac{a \cdot (x(Q)-b)}{x(Q)-a}, \\
\nonumber & & x(Q+P_3)=\frac{b \cdot (x(Q)-a)}{x(Q)-b}.
\end{eqnarray}
Note that a point $Q \in E(\bar{K})$ is in the kernel of reduction if
and only if $v \big( x(Q) \big) <0$.
It follows by the equations (\ref{coord1})
that $v \big( x(Q+P_i) \big) \geq 0$
for $Q \in S$ and $i=1,2,3$.
As a consequence we get
$h(x),g_i(x) \in R[x]$ for $i=1,2,3$.
\newline\indent
We claim that the isogeny $F$ reduces to the relative Frobenius
morphism.
The congruences
\begin{eqnarray}
\label{redxy}
\tilde{x}(x,y) \equiv x^p \bmod p \quad \mbox{and} \quad
\tilde{y}(x,y) \equiv y^p \bmod p
\end{eqnarray}
imply that
\[
a^{(p)} \equiv a^p \bmod p \quad \mbox{and} \quad
b^{(p)} \equiv b^p \bmod p.
\]
By the equations (\ref{coord1}) and $v \big( x(Q) \big) <0$
we have $v \big( x(Q + P_1) \big) >0$ for $Q \in S$.
It follows that
\[
g_1(x)^2 \equiv x^{p-1} \bmod p.
\]
Since for $Q \in S$ we have $v \big( x(Q) \big) <0$, it follows
by the definition of $h(x)$ that
\[
h(x) \equiv 1 \bmod p.
\]
We claim that
\begin{eqnarray}
\label{strangeq}
\label{imp}
\prod_{i=1}^3  \big( x-x(Q + P_i) \big) \equiv y^2
\bmod p.
\end{eqnarray}
Let $Q \in S$.
By the equations
(\ref{coord1}) and $v \big( x(Q) \big) <0$ we have
\begin{eqnarray*}
& & x(Q + P_1) \equiv 0 \bmod p, \\
& & x(Q+P_2) = \frac{ax(Q) -ab }{x(Q)-a}
= \frac{ a - \frac{ab}{x(Q)}}{1-\frac{a}{x(Q)}} \equiv a \bmod p
\end{eqnarray*}
and analogously
\begin{eqnarray*}
x(Q+P_3) \equiv b \bmod p.
\end{eqnarray*}
This proves the congruence (\ref{imp}). We conclude that
the congruences (\ref{redxy}) hold. Thus our claim is proven.
Beside that, the above discussion shows
that $a^{(p)}$ and $b^{(p)}$
are well-defined and
that $E^{(p)}$ is an elliptic curve.
\newline\indent
Finally, we claim that
\[
F^* \left( \frac{dx}{y} \right) = \mathrm{lead}(h) \cdot \frac{dx}{y}.
\]
We set
\[
f(x)=\frac{h(x) g_1(x) + 2x
\big( h(x) g_1'(x) - h'(x) g_1(x) \big)}{g_2(x)g_3(x)}.
\]
One computes that
\[
F^* \left( \frac{dx}{y} \right) = f(x) \cdot \frac{dx}{y}.
\]
Since $\frac{dx}{y}$ comes from
a global regular differential we deduce that $f(x)$ is constant.
We have
\[
f(0)=\frac{h(0) \cdot g_1(0)}{g_2(0) \cdot g_3(0)}
=\frac{g_1(0)}{g_2(0) \cdot g_3(0)}.
\]
Recall that $h(x)$ is normalized with respect to its constant term.
The formulae (\ref{coord1}) imply that
\[
x(Q+P_1) \cdot x(Q)=x(Q+P_2) \cdot x(Q+P_3).
\]
The claim now follows from the definition of $h(x)$ and $g_i(x)$,
where $i=1,2,3$.
This finishes the proof of Theorem \ref{froblift}.
\end{proof}
\noindent
Formulae of the same kind, but for separable isogenies, can be found
in \cite{ve71}.

\subsection{Special case: Residue class field of characteristic two}
\label{agm}

In this section we recall some results which are due to J.-F. Mestre
(see the private communication \cite{me00}).
For lack of a suitable reference we provide
proofs where necessary.
Mestre pointed out that a Frobenius lift in characteristic $2$
can be described by the classical arithmetic geometric
mean formulae.
This is explained in the following.
\newline\indent
Let $\xf_q$ be a finite field of characteristic $2$
and let $\xz_q$ denote the ring of Witt vectors
with values in $\xf_q$.
The field of fractions of $\xz_q$ will be denoted by $\xq_q$.
Let $E$ be a smooth elliptic curve over $\xz_q$, in other words
an abelian scheme of relative dimension $1$
over $\xz_q$.
\begin{proposition}
\label{ordellcurv}
We have $E[2] \cong \mu_{2,\xz_q} \times (\xz/ 2 \xz)_{\xz_q}$
if and only if
$E_{\xq_q}$ can be given by an equation of the form
\begin{eqnarray}
\label{defining}
y^2=x(x-a^2)(x-b^2),
\end{eqnarray}
where $a,b \in \xq_q^*$ such that $a \not= \pm b$, the point $(0,0)$ generates
$E[2]^{\mathrm{loc}}(\xq_q)$
and $\frac{b}{a} \in 1+ 8 \xz_q$.
\end{proposition}
\begin{prf}{Proof of Proposition \ref{ordellcurv}}
To prove the proposition we need the following fact.
\begin{claim}
\label{ordell}
Let $E$ be an elliptic curve over $\xz_q$ with $E[2] \cong \mu_{2,\xz_q}
\times (\xz / 2 \xz)_{\xz_q}$.
Then $E_{\xq_q}$ can be given by a model
\begin{eqnarray}
\label{startwo}
y^2=x(x-\alpha)(x-\beta), \quad \alpha,\beta \in \xq_q^*, \quad \alpha \not= \beta,
\end{eqnarray}
where $(0,0)$ generates $E[2]^{\mathrm{loc}}(\xq_q)$ and
$\frac{\beta}{\alpha} \in 1 + 16R$.
\end{claim}
\begin{prf}{Proof of Claim \ref{ordell}}
We can assume that $E$ is given by the equation (\ref{startwo})
and $(0,0)$ generates $E[2]^{\mathrm{loc}}(\xq_q)$.
We set $\lambda=\frac{\beta}{\alpha}$. One casn assume that
$\lambda \in \xz_q$. Then
\begin{eqnarray}
\label{jinv}
j(E_{\xq_q})=2^8 \frac{\big( (\lambda-1)^2 + \lambda \big)^3}{\lambda^2
(\lambda -1 )^2}.
\end{eqnarray}
Let $v$ denote the discrete exponential valuation of $\xq_q$
which satisfies $v(2)=1$.
Since $E$ has ordinary good
reduction one has $j(E) \not\equiv 0 \bmod 2$.
Hence equation (\ref{jinv}) implies that
\begin{eqnarray}
\label{wichtig}
0=8 + 3 v((\lambda-1)^2 + \lambda)-2v(\lambda)-2v(\lambda-1).
\end{eqnarray}
An isomorphism to a Weierstrass minimal model is of the form
\[
(x,y) \rightarrow (u^2x+r, \ldots).
\]
We can assume that the discriminant of the given model is a unit
and hence $u \in \xz_q^*$.
Since $(0,0)$ is in the kernel of reduction it follows that
$v(r)<0$. Also we have $v(u^2 \alpha + r) \geq 0$ and $v(u^2 \beta + r) \geq
0$,
because the points $(\alpha,0)$ and $(\beta,0)$ are not
contained in the kernel
of reduction.
We conclude that $v(\alpha)=v(\beta)=v(r)$ which implies $v(\lambda)=0$.
By (\ref{wichtig}) we cannot have $v(\lambda-1)=0$.
Hence $v(\lambda-1)>0$, and it follows again by (\ref{wichtig})
that $v(\lambda-1)=4$. This implies the claim.
\end{prf}
\noindent
Now we finish the proof of Proposition \ref{ordellcurv}.
Let $E$ be as in the proposition and let
\[
y^2=x(x-\alpha)(x-\beta), \quad \alpha,\beta \in K^*, \quad \alpha \not=\beta,
\]
be a model for $E$ over $\xq_q$ having the properties listed in
Claim \ref{ordell}.
Over $L=\xq_q(i)$ the above curve is isomorphic
to the twisted curve $E^t$
given by
\[
y^2=x(x+\alpha)(x+\beta)
\]
via the isomorphism
\begin{eqnarray}
\label{twistiso}
(x,y) \mapsto (-x,iy).
\end{eqnarray}
Now by \cite[Ch.X,Prop.1.4]{si86} we have an equivalence
\begin{eqnarray}
\label{quadrat}
\mbox{$\alpha,\beta$ squares in $\xq_q^*$} \; \Leftrightarrow \;
     [2]^{-1}_{E^t}(0,0)(\xq_q) \not= \emptyset.
\end{eqnarray}
We claim that the right hand side of (\ref{quadrat}) holds.
Let $G_L = \mbox{Gal}(L/ \xq_q)$.
The isomorphism (\ref{twistiso}) induces an isomorphism of groups
$E[4](L) \stackrel{\sim}{\rightarrow} E^t[4](L)$.
One computes $\sigma(P^t)=-(\sigma(P))^t$ for $\mbox{id} \not= \sigma \in G_L$.
As a consequence
$\sigma(P^t)=P^t$ if and only if $\sigma(P)=-P$.
Hence we have
\begin{eqnarray}
\label{lasting}
[2]^{-1}_{E}(0,0)(\xq_q) = \emptyset \implies [2]^{-1}_{E^t}(0,0)(\xq_q) \not=
\emptyset.
\end{eqnarray}
Suppose $P \in [2]^{-1}_{E}(0,0)(\xq_q)$.
Let $Q$ be a point of order $2$
which does not lie in the kernel of reduction. Then
\[
[2]^{-1}_{E}(0,0)(\xq_q) = \{ P, -P, P+Q, -(P+Q) \}.
\]
Two of these four points have to be in the kernel of reduction.
Thus the points of $E[4]^{\mathrm{loc}}$ are rational over $\xq_q$.
This implies $i \in \xq_q$ which is a contradiction.
Since the converse direction in Proposition \ref{ordellcurv}
is trivial, this finishes the proof.
\end{prf}
\noindent
Assume now
that $E$ satisfies the equivalent conditions of Proposition
\ref{ordellcurv} and let $E_{\xq_q}$ be given by equation (\ref{defining}).
By our assumption $E$ has ordinary reduction.
The condition $\frac{b}{a} \in 1 + 8 \xz_q$ implies that $\frac{b}{a}$ is a
square in $\xz_q$.
We set in analogy to the classical AGM formulae
\begin{eqnarray}
\label{classicalform}
\tilde{a}=\frac{a+b}{2}, \quad
\tilde{b}=\sqrt{ab}=a \sqrt{\frac{b}{a}},
\end{eqnarray}
where we choose $\sqrt{\frac{b}{a}} \in 1 +4 \xz_q$.
\begin{proposition}
\label{agmprop}
Let $E^{(2)}$ be defined as in Section \ref{sectconv}.
The curve $E_{\xq_q}^{(2)}$ admits the model
\begin{eqnarray}
\label{tileq}
y^2=x(x-\tilde{a}^2)(x-\tilde{b}^2)
\end{eqnarray}
where the point $(0,0)$ generates $E^{(2)}[2]^{\mathrm{loc}}(\xq_q)$
and $\frac{\tilde{b}}{\tilde{a}} \in 1 +8 \xz_q$.
The isogeny $F_{\xq_q}:E_{\xq_q} \rightarrow E_{\xq_q}^{(2)}$ is given by 
\begin{eqnarray*}
(x,y) \mapsto \left(
\frac{(x+ab)^2}{4x},
\frac{y(ab-x)(ab+x)}{8x^2}
\right).
\end{eqnarray*}
\end{proposition}
\begin{proof}
It is straight forward to verify that
the curve $E_{\xq_q}$ is isogenous with the elliptic curve
\begin{eqnarray}
\label{secondcurve}
y^2=x \cdot \left( x- \left( \frac{a-b}{2} \right)^2 \right)
\cdot \left( x- \left( \frac{a+b}{2} \right)^2 \right)
\end{eqnarray}
via the isogeny
\begin{eqnarray}
\label{twoisog}
(x,y) \mapsto \left( \frac{y^2}{4x^2}, \frac{y(ab-x)(ab+x)}{8x^2}
\right)
\end{eqnarray}
which has kernel equal to
\[
E[2]^{\mathrm{loc}}(\xq_q)= \langle ( 0,0) \rangle.
\]
The point $(0,0)$ on the curve defined by (\ref{secondcurve})
is not in the kernel of reduction, because it is
the image of the point $(a^2,0)$ or $(b^2,0)$, where each of the latter
points induces a non-trivial point in $E[2]^{\mathrm{et}}(\xq_q)$.
Consider the $x$-coordinates
\[
\left( \frac{a-b}{2} \right)^2  \quad \mbox{and}
\quad \left( \frac{a+b}{2} \right)^2
\]
of the other two $2$-torsion points.
The one with the smaller valuation is 
the $x$-coordinate of the $2$-torsion point in the kernel of
reduction, because an
isomorphism over $\xq_q$ to a minimal model
preserves the ordering given by the valuations.
Let $v$ be a discrete exponential valuation of $\xq_q$ such that $v(2)=1$.
Since $b/a \in 1 + 8 \xz_q$ it follows that
\[
v(a+b)=v(a)+v(1+\frac{b}{a})=v(a)+1 < v(a) +
v(1-\frac{b}{a}) = v(a-b).
\]
The transformation
\begin{eqnarray}
\label{trafo}
(x,y) \mapsto \left( x - \left( \frac{a+b}{2} \right)^2 , y \right)
\end{eqnarray}
yields the model (\ref{tileq}).
The proposition now follows by composing the morphisms
(\ref{twoisog}) and (\ref{trafo}). 
\end{proof}

\section{Torsion points on ordinary elliptic curves}
\label{divpoly}

Let $K$ be a field of characteristic $\not=2$.
Suppose we are given an elliptic curve $E$ over $K$ by an equation
\begin{eqnarray}
\label{definingeq}
y^2=x(x-a)(x-b)
\end{eqnarray}
where $a,b \in K^*$ and $a \not=b$.
In this
section we introduce the so-called division polynomials, which
describe the torsion points of $E$.
\begin{definition}
Let
\begin{eqnarray*}
& & \psi_0=0, \; \psi_1=1, \; \psi_2=2y \\
& & \psi_3=3x^4-4(a+b)x^3+6abx^2-(ab)^2 \\
& & \psi_4=2y \big( 2x^6-4(a+b)x^5+10abx^4-10(ab)^2x^2 \\
& & \quad +4(ab)^2(a+b)-2(ab)^3 \big) \\
& & \ldots \\
& & \psi_{2l+1}=\psi_{l+2}\psi_l^3-\psi_{l-1}\psi^3_{l+1}, \; l \geq
2, \\ 
& &
\psi_{2l}= \frac{\psi_l}{2y}(\psi_{l+2}\psi^2_{l-1}-\psi_{l-2}\psi^2_{l+1}),
\; l>2.
\end{eqnarray*}
The polynomial $\psi_m(x,y)$, $m \geq 0$, is called the $m$-th
division polynomial of $E$.
\end{definition}
\noindent
The polynomial $\psi_m(x,y)$ defines a function on $E$.
Let $\bar{K}$ denote an algebraic closure of $K$.
The following proposition is classical.
\begin{proposition}
\label{classical}
Let $m \geq 2$.
Assume that $K= \bar{K}$. Then the function
\[
\psi_m:E \rightarrow \xp^1_K, (x,y) \mapsto \psi_m(x,y)
\]
has divisor
\[
\sum_{P \in E[m](K)} 
\mathrm{deg}_{\mathrm{i}}[m] \cdot (P) \quad  - \quad m^2 \cdot (0_E),
\]
where $0_E$ denotes the point at infinity and $\mathrm{deg}
_{\mathrm{i}}[m]$
the degree of inseparability of the isogeny $[m]:E \rightarrow E$.  
\end{proposition}
\noindent
In the following we assume that $m \geq 3$ is an odd integer.
By induction one proves that
the variable $y$ in $\psi_m(x,y)$ occurs only with even
exponent.
Substituting successively $y^2$ by $x(x-a)(x-b)$
we get a polynomial in the variable $x$.
We denote the resulting polynomial by $\psi_m(x)$.
Choose $S \subseteq E[m](\bar{K})$ such that $S \cap -S = \emptyset$ and
$E[m](\bar{K}) = S \cup -S \cup \{ 0_E \}$.
\begin{corollary}
\label{super}
There exists a constant $c \in K$ such that 
\begin{eqnarray}
\label{divpolyeq}
\psi_m(x)= c \prod_{P \in S} \left( x-x(P) \right)^{\mathrm{deg}_i[m]}
\end{eqnarray}
where $x(P)$ denotes the $x$-coordinate of $P$.
If $\mathrm{deg}_i[m]=1$, then
\[
\mathrm{deg}(\psi_p)=\frac{m^2-1}{2} \quad \mbox{and} \quad c=m.
\]
\end{corollary}
\begin{proof}
The first claim follows from Proposition \ref{classical},
since the function on the right hand side of equation
(\ref{divpolyeq}) has the same divisor as $\psi_m(x)$.
Now assume that $\mathrm{deg}_i[m]=1$. This implies that
$\psi_m(x)$ has degree $(m^2-1)/2$. Using induction on $m$ one shows
that the coefficient of $x^{(m^2-1)/2}$ in $\psi_m(x)$ equals $m$.
This implies the second claim.
\end{proof}
\noindent
Now let $\xf_q$ be a finite field of characteristic $p>2$.
Let $\xz_q$ be the ring of Witt vectors
with values in $\xf_q$ and let $\xq_q$ be the field of fractions of $\xz_q$.
Assume that $a,b \in \xz_q^*$ and $a \not\equiv b \bmod p$, where $a$
and $b$ are the coefficients of the elliptic curve (\ref{definingeq}).
By our assumption, the model (\ref{definingeq}) is a Weierstrass
minimal model and $E$ has good reduction.
Let $v_p$ be an additive discrete valuation of $\xq_q$
which is assumed to be normalized such that $v_p(p)=1$.
\begin{proposition}
\label{newton}
The curve $E$ has ordinary reduction if and only if
the Newton polygon of  $\frac{1}{p}\psi_p(x)$ with
respect to $v_p$ is as in \rm{Figure \ref{fig:newton}}.
\end{proposition}
\begin{proof}
Let $\psi_p(x) \in \xz_q[x]$ denote the $p$-th division polynomial on $E$
and $\bar{\psi}_p(x) \in \xf_q[x]$ its reduction modulo $p$.
The degree of inseparability of $[m]$ on $E$
equals $1$. By Corollary \ref{super} the polynomial
$\psi_p(x)$ has degree $(p^2-1)/2$ and leading coefficient $p$.
We choose an extension of $v_p$ to $\bar{\xq}_q$, that we denote by
the same symbol $v_p$.
Note that $Q \in E(\bar{\xq}_q)$ is in the kernel of reduction
if and only if $v_p \big(  x(Q) \big)<0$.
Since $p>2$, the $x$-coordinates of points in $E[p](\bar{\xq}_q)$,
which are not in the kernel of reduction,
have valuation $0$.
\newline\indent
Suppose that $E$ has ordinary reduction.
This means that $\#\bar{E}(\bar{\xf}_q)=p$ and the
degree of inseparability of the isogeny $[p]$ on $\bar{E}$
equals $p$ (compare \cite{si86} Ch.III, Corollary 6.4).
By the above discussion the Newton polygon of
$\frac{1}{p}\psi_p(x)$ has a segment of slope $0$ and length $p(p-1)/2$.
Also it has a segment of strictly positive slope, which has
length $(p-1)/2$ and corresponds to the points of
$E[p](\bar{\xq}_q)$ lying in the kernel of reduction.
Since by Corollary \ref{super}
the leading coefficient of $\psi_p(x)$ equals $p$ and $v_p$ is integer
valued on $\xq_q$, we conclude that the constant term of
$\psi_p(x)$ has valuation $0$ and the segment
of strictly positive slope of the Newton polygon of $\frac{1}{p}\psi_p(x)$
is a straight line.
\newline\indent
Suppose that $E$ has supersingular reduction.
Then $\#\bar{E}(\bar{\xf}_q)=1$.
By Corollary \ref{super} the reduced polynomial $\bar{\psi}_p(x)$
equals a constant. This implies that the Newton polygon of
$\frac{1}{p}\psi_p(x)$ has strictly positive slope.
This finishes the proof of the proposition. 
\end{proof}
\noindent
For more details about division polynomials we refer to
\cite{sl78} Ch.II, \cite{ca66} and \cite{ay92}.
\begin{corollary}
\label{subgroup}
Let $E$ have ordinary reduction. Then there exists a subgroup
$G \leq E[p](\bar{\xq}_q)$ defined over $\bar{\xq}_q$, which is uniquely
determined by the conditions that it is of order $p$ and
lies in the kernel of reduction.
\end{corollary}
\begin{proof}
Assume that $E$ is an elliptic curve which has ordinary reduction.
By Proposition \ref{newton} it follows that there are precisely $p$
points of order $p$ on $E$ lying in the kernel of reduction.
Let $0_E \not= P \in E[p](\bar{\xq}_q)$ be in the kernel of reduction.
Then the multiples of $P$ are as well, since the reduction map is a
homomorphism of groups. This proves the corollary.
\end{proof}
\begin{figure}[htb]
\begin{center}
\begin{picture}(0,0)%
\includegraphics{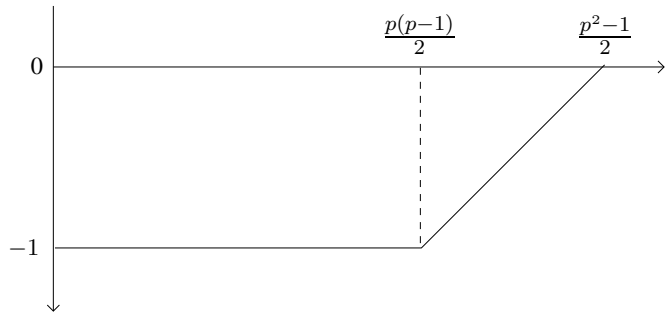}
\end{picture}%
\setlength{\unitlength}{2901sp}%
\begingroup\makeatletter\ifx\SetFigFont\undefined%
\gdef\SetFigFont#1#2#3#4#5{%
\reset@font\fontsize{#1}{#2pt}%
\fontfamily{#3}\fontseries{#4}\fontshape{#5}%
\selectfont}%
\fi\endgroup%
\begin{picture}(5605,2634)(1081,-2188)
\put(5941,119){\makebox(0,0)[lb]{\smash{\SetFigFont{11}{13.2}{\rmdefault}{\mddefault}{\updefault}$\frac{p^2-1}{2}$}}}
\put(4276,119){\makebox(0,0)[lb]{\smash{\SetFigFont{11}{13.2}{\rmdefault}{\mddefault}{\updefault}$\frac{p(p-1)}{2}$}}}
\put(1081,-1681){\makebox(0,0)[lb]{\smash{\SetFigFont{9}{10.8}{\rmdefault}{\mddefault}{\updefault}$-1$}}}
\end{picture}    
\caption{Newton polygon in case of ordinary reduction}
\label{fig:newton}
\end{center}
\end{figure}
\noindent
Combining Theorem \ref{froblift} and Corollary \ref{subgroup}
we get an elementary proof for the existence of a lift of relative Frobenius
in the case of ordinary reduction.

\section{Algorithmic aspects of Frobenius lifting}
\label{alggo}

In this section we explain how one can apply the results
of the previous sections in order to count points
on elliptic curves over finite fields. 

\subsubsection*{Notation}

We first fix some notation that will be used in the following
sections.
Let $\xz_q$ denote the Witt vectors with values
in a finite field $\xf_q$ with $q=p^d$
elements, where $p$ is a prime.
We say that an element $x \in \xz_q$ is given with precision $m$
if it is given modulo $p^m$. One can carry out arithmetic
operations with precision $m$ by considering the given
quantities as elements of the quotient ring $\xz_q/(p^m)$.
For the implementation of the arithmetic in $\xz_q/(p^m)$
see \cite{fgh00} $\S$2.

\subsection{Computing a Frobenius lift in the ordinary case}
\label{liftalgo}

Let now $\xz_q$ denote the Witt vectors with values
in a finite field $\xf_q$ with $q=p^d$
elements, where $p>2$ is a prime.
Let $E$ be an elliptic curve
which is given by the equation
\begin{eqnarray*}
y^2=x(x-a)(x-b)
\end{eqnarray*}
where $a,b \in \xz_q^*$ and $a \not\equiv b \bmod p$.
We assume that $E$ has ordinary good reduction.
By Corollary \ref{subgroup} and Theorem \ref{froblift} there
exists an explicit Frobenius lift $F:E \rightarrow E^{(p)}$.
Let $a^{(p)}$ and $b^{(p)}$ be defined as in Theorem
\ref{froblift}.
\begin{theorem}
\label{comp}
One can give a deterministic algorithm, which has as input the coefficients
$a$ and $b$ of $E$ with precision $m$ and as output the coefficients
$a^{(p)}$ and $b^{(p)}$ of $E^{(p)}$ with precision $m$, 
such that its complexity equals $O \big( p^{2+\epsilon}
(dm)^{1+\epsilon} \big)$ for all $\epsilon>0$.
\end{theorem}
\begin{proof}
In the following we give the algorithm, whose existence
is claimed in the theorem.
By $\psi_p(x)$ we denote the $p$-th division polynomial corresponding
to the points of order $p$ on $E$
(compare Section \ref{divpoly}).
The algorithm is as follows.
\begin{algorithm}
\label{fl}
Input: $a,b \in \xz_q/(p^m)$; Output: $a^{(p)},b^{(p)} \in \xz_q/(p^m)$
\begin{enumerate}
\setlength{\itemsep}{0.2cm}
\item
\label{een}
Compute $\psi_p(x) \bmod p^m$.
\item
\label{drie}
Find a decomposition
\begin{eqnarray}
\label{dd}
x^{\frac{p^2-1}{2}} \cdot \psi_p \left( \frac{1}{x} \right)
\equiv U(x) \cdot W(x) \bmod p^m
\end{eqnarray}
where $W(x)$ is monic and $W(x) \equiv x^{\frac{p-1}{2}} \bmod p$
using Hensel's algorithm.
Set
\[
V(x)=x^{\frac{p-1}{2}} \cdot W \left( \frac{1}{x} \right).
\]
\item
\label{fijf}
Compute
\[
a^{(p)}=a^p \cdot \left( \frac{V(b)}{V(a)}
\right)^2 \bmod p^m
\quad
\mbox{and}
\quad
b^{(p)}=b^p \cdot \left( \frac{V(a)}{V(b)}
\right)^2 \bmod p^m.
\]
\end{enumerate}
\end{algorithm}
\noindent
First we prove the correctness of Algorithm \ref{fl}.
The $p$-th division polynomial $\psi_p(x)$ is computed in Step \ref{een} with
precision $m$ using the formulae of Section \ref{divpoly}.
Since $E$ has ordinary reduction, it follows by Proposition
\ref{newton} that the polynomial
$x^{(p^2-1)/2} \cdot \psi_p(1/x)$ reduces modulo $p$ to a polynomial
of degree $(p^2-1)/2$ which is divisible by $x^{(p-1)/2}$ and not
divisible by a bigger power of $x$.
By Hensel's Lemma \cite[Kap.II,Lem.4.6]{ne92} one can find a decomposition
of the form (\ref{dd}) lifting the factor $x^{(p-1)/2}$.
The latter composition is computed in Step \ref{drie}.
By construction,
the factor $V(x)$ corresponds to a subgroup $G \leq E[p](\bar{K})$ of
order $p$ contained in the kernel of reduction.
One can apply Theorem \ref{froblift} to $G$ in order
to compute the coefficients $a^{(p)}$ and $b^{(p)}$ of
the curve $E^{(p)}$. This is done in Step \ref{fijf}.
The polynomial $V(x)$ differs multiplicatively from $h(x)$ by
a unit. Thus we have
\[
\frac{V(b)}{V(a)} \equiv \frac{h(b)}{h(a)} \bmod p^m.
\]
This proves the correctness of the Algorithm \ref{fl}.
\newline\indent
Next we provide some well-known results about the complexity of
the arithmetic operations in the Witt vectors of a finite field.
Elements of $\xz_q/(p^m)$ allocate
$O \big( md \log_2(p) \big)$
bits if one stores them as integers.
For details see \cite{fgh00} $\S$2.
Using fast integer multiplication techniques we conclude
that a multiplication in $\xz_q/(p^m)$
has complexity
$O \big( (md)^{1+ \epsilon} \log_2(p)^{1+ \epsilon} \big)$.
Inversion of $a \in \xz_q /(p^m)$ can be done using Newton iteration
applied to the polynomial $ax-1 \in \xz_q/(p^m)[x]$. The complexity of the
Newton iteration is analyzed in \cite{fgh00} $\S$2.5. The resulting
complexity of an inversion is equal to that of the multiplication.
Representing elements of $\xz_q/(p^m)[x]$ as integers and using a
fast arithmetic for integers the complexity
of the multiplication of two polynomials in
$\xz_q/(p^m)[x]$ of degree $n$
becomes $O \big( (nmd)^{1+ \epsilon} \log_2(p)^{1+ \epsilon} \big)$.
In order to prove the complexity bound of Theorem \ref{comp}
we will analyze step-by-step the relevant parts of Algorithm
\ref{fl}.
\newline\newline
{\bf Step \ref{een}:} For the computation of the division polynomial $\psi_p$ we
use the formulae of Section \ref{divpoly}.
Note that for every $m \geq 5$
the polynomial $\psi_m$ can be computed in terms of polynomials
forming a subset of the set
\begin{eqnarray}
\label{interval}
\{ \psi_{n+2}, \ldots, \psi_{n-2} \},
\end{eqnarray}
where $n=\lfloor m/2 \rfloor$.
This shows that a recursive algorithm for computing $\psi_p$ has depth
$\lfloor \log_2(p) \rfloor$.
The necessary polynomial multiplications to compute
$\psi_m$ in terms of the polynomials (\ref{interval})
can be performed in
$O \big( n^{2 + \epsilon} (md)^{1+ \epsilon} \log_2(p)^{1+ \epsilon} \big)$
bit operations, since $\psi_n$ has degree $(n^2-1)/2$.
Let $s_1=\lceil p/2 \rceil +2$ and
$t_1=\lfloor p/2 \rfloor -2$.
We set
\[
s_i= \lceil \frac{s_{i-1}}{2} \rceil + 2
\quad \mbox{and} \quad t_i = \lfloor \frac{t_{i-1}}{2} \rfloor -2
\]
for $i>1$.
In our case, we have to compute the polynomials
$\psi_{s_i}, \ldots , \psi_{t_i}$
for $i \geq 1$.
By induction on $i$ one can prove that
\[
s_i \leq \lceil \frac{p}{2^i} \rceil + (i-1) + 2 \quad \mbox{and}
\quad t_i \geq \lfloor \frac{p}{2^i} \rfloor - (i-1) - 2.
\]
It follows that the number of polynomials to be computed on each
recursion level grows
linearly in the index $i$.
Since $i \leq \lceil \log_2(p) \rceil$ we
conclude that the $p$-th division polynomial $\psi_p$
can be computed in $O \big( p^{2+ \epsilon} (md)^{1+ \epsilon} \big)$
bit operations.
\newline\newline
{\bf Step \ref{drie}:}
Using the standard Hensel algorithm (see \cite{co93} Section 3.5.3)
we obtain for the second step in the algorithm
the complexity $O \big( p^{2+\epsilon} (md)^{1+ \epsilon} \big)$.
Note that Hensel's algorithm converges quadratically.
We assume that one uses in each iteration the minimal precision
required in order to get the correct result.
\newline\newline
{\bf Step \ref{fijf}:} Evaluating a polynomial in $\xz_q/(p^m)[x]$ of degree $(p-1)/2$
at a value in $\xz_q/(p^m)$ has complexity
$O \big( (pmd)^{1+\epsilon} \big)$.
To achieve this complexity one uses a squaring table
and a $2$-adic representation of exponents.
We do not describe this method in detail because it is
standard.
\newline\newline
Summing up the above complexities we get the complexity
bound as stated in Theorem \ref{comp}.
\end{proof}

\subsection{Generalizing Mestre's AGM algorithm}
\label{pointcounting}

Let $\bar{E}$ be an ordinary elliptic curve over a finite field $\xf_q$
of characteristic $p>2$ given by the Weierstrass equation
\[
y^2=x(x-\bar{a})(x-\bar{b})
\]
where $\bar{a},\bar{b} \in \xf_q$.
\begin{theorem}
\label{pccomp}
The Algorithm \ref{countpoints}, which has as input a finite field
$\xf_q$ and the coefficients $\bar{a},\bar{b}$ of $\bar{E}$,
computes the number of $\xf_q$-rational points
$\# \bar{E}(\xf_q)$ on $\bar{E}$ in time
$O \big( p^{2+\epsilon}d^{3+\epsilon} \big)$, where $d=\log_p(q)$,
for all $\epsilon >0$.
\end{theorem}
\noindent
We note that it is straight forward to modify the Algorithm
\ref{countpoints} such that one can drop the assumption
that $\#\bar{E}[2](\xf_q)=4$.
In the following let $\xz_q$ denote the Witt vectors
with values in $\xf_q$.
\begin{algorithm}
\label{countpoints}
Input: $\bar{a},\bar{b} \in \xf_q$ ; Output: $\#\bar{E}(\xf_q)$
\begin{enumerate}
\setlength{\itemsep}{0.2cm}
\item
\label{stepeins}
Set
\[
d=\log_p(q), \quad \mbox{and} \quad m=d+\lceil d/2 \rceil +2.
\]
\item
\label{stepzwei}
Choose $a,b \in \xz_q/(p^m)$ such that
\[
a \equiv \bar{a} \bmod p \quad \mbox{and} \quad b \equiv \bar{b} \bmod
p.
\]
\item
\label{stepdrei}
Compute the pairs
\[
(a,b),(a_1,b_1), \ldots , (a_{m-1},b_{m-1}),
\]
where $a_i=a^{(p^i)}$ and $b_i=b^{(p^i)}$,
by iterating $(m-1)$-times the Algorithm \ref{fl} with precision $m$.
\item
\label{stepvier}
Compute with precision $m$ the triples
\[
(a_m,b_m,c_m), \ldots, (a_{m+d-1},b_{m+d-1},c_{m+d-1})
\]
where $a_i,b_i$ are as in Step \ref{stepdrei}
and $c_i$ is computed as follows:
We use a modified version of Algorithm \ref{fl} having as
additional output the number
\[
c=\mathrm{lead} (V) \bmod p^m,
\]
where $\mathrm{lead}(V)$ denotes the leading coefficient of
the polynomial $V(x)$ which is computed in Step \ref{drie}
of Algorithm \ref{fl}.
\item
\label{stepfuenf}
Compute
\[
v= \left( \frac{\prod_{i=m}^{m+d-1} c_i}{p^d} \right)^{-1} \bmod p^{m-d}.
\]
\item
\label{stepsechs}
Compute
\[
t = v + \frac{q}{v} \bmod p^{m-d}.
\]
Find the unique integer $t_0$ in the interval
$[q+1-2\sqrt{q}, \ldots, q+1+2\sqrt{q}]$
such that $t_0 \equiv t \bmod p^{m-d}$.
Return $q+1-t_0$.
\end{enumerate}
\end{algorithm}
\begin{prf}{Proof of Theorem \ref{pccomp}}
First we prove the correctness of
Algorithm \ref{countpoints}.
Morally, the correctness of the algorithm follows
from the following observation, which
can be explained using Corollary \ref{conv}.
\begin{fact}
Let $E^*$ be the canonical lift of $\bar{E}$ and let
$E$ be defined by the equation
\begin{eqnarray}
\label{deq}
y^2=x(x-a)(x-b), \quad a,b \in \xz_q^*,
\end{eqnarray}
such that $\bar{a} \equiv a \bmod p$
and $\bar{b} \equiv b \bmod p$.
If we define $E^{(p^i)}$ for $i \geq 1$ as in Section \ref{sectconv}
then one has
\[
\lim_{n \rightarrow \infty} j \big( E^{(q^n)} \big)
= j \big( E^* \big)
\]
with respect to the $p$-adic topology.
\end{fact}
\noindent
Let $d$ and $m$ be as in Step \ref{stepeins}.
The choice of $a$ and $b$ in Step \ref{stepzwei}
determines an elliptic curve $E$ with defining
equation of the form (\ref{deq}).
Now let $E^{(p^i)}$ for $i \geq 1$ be defined
as in Section \ref{sectconv}.
We note that by Corollary \ref{conv} the curve $E^{(p^{m-1})}$
with the coefficients $a^{(p^{m-1})}$ and $b^{(p^{m-1})}$, 
which are computed in Step \ref{stepdrei}, is the
canonical lift over $R/(p^m)$ of its reduction.
The latter is also true for the curve
$E^{(p^{m+d-1})}$ with
coefficients $a^{(p^{m+d-1})}$ and $b^{(p^{m+d-1})}$,
which are computed in Step \ref{stepvier}.
The reductions of $E^{(p^{m-1})}$ and $E^{(p^{m+d-1})}$
coincide.
As a consequence there exists a unique isomorphism
\[
\varphi:E^{(p^{m+d-1})} \stackrel{\sim}{\longrightarrow} E^{(p^{m-1})}
\]
defined over $R/(p^m)$ such that
the composed map
\[
\Phi=\varphi \circ F^d \in \mathrm{End}_{R/(p^m)} \big( E^{(p^{m-1})} \big)
\] 
reduces to the absolute Frobenius of the reduction
$\bar{E}^{(p^{m-1})}$.
The bijectivity of the reduction map on homomorphisms
implies that $\varphi= \mathrm{id}$ and thus
\[
a^{(p^{m+d-1})} \equiv a^{(p^{m-1})} \bmod p^m
\quad \mbox{and} \quad
b^{(p^{m+d-1})} \equiv b^{(p^{m-1})} \bmod p^m.
\]
The map
\begin{eqnarray}
\label{bij}
\mathrm{End}_{R/(p^m)} \big( E^{(p^{m-1})} \big) 
\rightarrow \mathrm{End}_{\xf_q} \big( \bar{E}^{(p^{m-1})} \big)
\end{eqnarray}
induced by reduction is bijective because of the
characterizing property of
the canonical lift.
Let $V=\hat{\Phi}$ be the dual of the isogeny $\Phi$. The isogeny $V$ lifts the
absolute Verschiebung morphism of $\bar{E}^{(p^{m-1})}$.
By the injectivity of (\ref{bij}) the equality
\begin{eqnarray}
\label{charpolyv}
V^2 - [t] \circ V + [q] = 0
\end{eqnarray}
holds in the ring
\[
\mathrm{End}_{R/(p^m)} \big( E^{(p^{m-1})} \big),
\]
where $t$ denotes the trace of the absolute $q$-Frobenius
morphism on $\bar{E}^{(p^{m-1})}$.
As a consequence of equation (\ref{charpolyv}) we get
\begin{eqnarray}
\label{sois}
\big( V^2 - [t] \circ V + [q] \big)^* \left( \frac{dx}{y} \right) = 0.
\end{eqnarray}
We define $v \in R/(p^m)$ by the equation
\[
V^* \left( \frac{dx}{y} \right) = v \cdot \frac{dx}{y}.
\]
Note that the Verschiebung is a separable isogeny acting as a
non-zero scalar on the differentials of $\bar{E}^{(p^{m-1})}$
over $\xf_q$. This shows that $v$ is invertible modulo $p^{m}$.
We remark that on the other hand
the scalar, which describes the action of $F$ on
differentials, is divisible by $q$. This is the reason why we
work with the isogeny $V$ instead of the isogeny $\Phi$.
We conclude from (\ref{sois}) that
\begin{eqnarray}
\label{congruence}
t \equiv v+\frac{q}{v} \bmod p^{m}.
\end{eqnarray}
The number of $\xf_q$-rational points on $\bar{E}$
equals that of $\bar{E}^{(p^{m-1})}$ since the two curves
are isogenous over $\xf_q$. This shows that the above number $t$
is in fact the trace of the absolute Frobenius on $\bar{E}$.
\newline\indent
In the following we describe the relevant steps
of Algorithm \ref{countpoints} in more detail
and give their complexity.
\newline
\newline
{\it Step \ref{stepeins}:} In order to turn the congruence
(\ref{congruence}) into an equality, which holds in $\xz$,
we have to choose the right precision.
Hasse's Theorem (see \cite{si86} Ch. V, Theorem 1.1) states that
\[
| t | \leq 2 \sqrt{q} < p^{\lceil d/2 \rceil + 1}.
\]
We conclude that one can
recover the value for $t$
from the approximation modulo $p^m$
if one takes $m= \lceil d/2 \rceil + 2$.
It will be explained in Step \ref{stepvier}
why we actually compute
with precision $d+\lceil d/2 \rceil + 2$.
\newline\newline
{\it Step \ref{stepdrei}:} One has to iterate $(m-1)$-times Algorithm
\ref{fl} with precision $m$.
The resulting complexity of Step \ref{stepdrei}
is $O(p^{2+\epsilon}d^{3+\epsilon})$ by Theorem
\ref{comp}.
\newline\newline
{\it Step \ref{stepvier}:}
Similar as in Step \ref{stepdrei} the overall complexity of
Step \ref{stepvier} is $O(p^{2+\epsilon}d^{3+\epsilon})$.
\newline\newline
{\it Step \ref{stepfuenf}:}
By Theorem \ref{froblift} the scalar $\prod_{i=m}^{m+d-1} c_i$
describes the action of the Frobenius morphism
on differentials.
To obtain the value $v$ as above, one has to divide the
product $\prod_{i=m}^{m+d-1} c_i$ by $p^d$. By doing so
one loses precision $d$.
This loss of precision is compensated
by performing all necessary computations
modulo $m$ where $m=d+\lceil d/2 \rceil + 2$.
\end{prf}

\section{Examples and practical results}
\label{exagm}

First, we illustrate the generalized AGM method by an
example.
Consider an elliptic curve of the form
\[
y^2=x(x-a)(x-b),
\]
over the integers of the degree $6$ unramified extension of $\xz_3$.
Assume that $a=1$ and $b$ is given be the congruence class of the
polynomial
\[
191096x^5 + 198863x^4 - 40571x^3 + 247894x^2 + 127753x + 193545
\]
in the quotient ring $\xz_3[x]/(f)$,
where $f=x^6 + 2x^4 + x^2 + 2x + 2$ and the
computing precision is $12$.
The generalized AGM sequence is given by the
sequence of elliptic curves $E_n$
of the form
\[
y^2=x(x-a_n)(x-b_n)
\]
where $a_n$ and $b_n$ are as in the following table.
Also we give the $j$-invariant $j_n$ of $E_n$.
\begin{eqnarray*}
\begin{array}{|l|c|}
\hline
n & a_n \\
\hline
1&-219543x^5 - 174456x^4 + 242538x^3 + 50793x^2 + 73503x - 114671
\\ \hline
2&244131x^5 + 164118x^4 + 59862x^3 + 81231x^2 + 5310x + 222361
\\ \hline
3&36027x^5 + 182667x^4 - 141981x^3 + 77385x^2 - 236172x - 16334
\\ \hline
4&-115041x^5 + 88929x^4 + 144273x^3 + 96438x^2 - 77580x - 52628
\\ \hline
5&-199374x^5 + 26007x^4 + 115827x^3 + 119622x^2 - 251307x + 89887
\\ \hline
6&111870x^5 + 262608x^4 - 100830x^3 - 12261x^2 - 165993x + 42697
\\ \hline
7&-137418x^5 - 174771x^4 + 117006x^3 - 114177x^2 - 30474x + 3949
\\ \hline
8&-88185x^5 - 220038x^4 + 18714x^3 + 254922x^2 + 197199x + 161044
\\ \hline
9&35946x^5 + 201945x^4 + 205590x^3 + 80220x^2 + 40443x + 94798
\\ \hline
10&103659x^5 - 29412x^4 + 229809x^3 - 168675x^2 + 206487x + 255010
\\ \hline
11&-163653x^5 + 125880x^4 - 159735x^3 - 90330x^2 - 131751x + 230584
\\ \hline
12&-98082x^5 + 30786x^4 - 30846x^3 - 217839x^2 - 262221x - 36035
\\ \hline
13&-163662x^5 + 48303x^4 - 263532x^3 - 173226x^2 + 245088x + 227023
\\ \hline
14&-9453x^5 - 259404x^4 + 136812x^3 - 79689x^2 - 157095x + 42946
\\ \hline
15&-200250x^5 + 260994x^4 - 89655x^3 + 21171x^2 - 254802x - 23300
\\ \hline
16&103659x^5 - 206559x^4 - 124485x^3 + 8472x^2 + 29340x + 255010
\\ \hline
17&-163653x^5 + 125880x^4 - 159735x^3 - 90330x^2 - 131751x + 230584
\\ \hline
\end{array}
\end{eqnarray*}
\begin{eqnarray*}
\begin{array}{|l|c|}
\hline
n&b_n \\
\hline
1&235370x^5 + 28234x^4 - 212531x^3 - 159624x^2 - 170578x + 222642
\\ \hline
2&179553x^5 - 220534x^4 + 163518x^3 + 137832x^2 + 144738x + 181163
\\ \hline
3&-6923x^5 + 12185x^4 - 178985x^3 - 215143x^2 - 105466x - 184699
\\ \hline
4&-107763x^5 - 219958x^4 - 210434x^3 + 237327x^2 - 82574x + 52910
\\ \hline
5&183685x^5 - 88458x^4 + 204515x^3 - 191205x^2 + 230688x - 264484
\\ \hline
6&-49681x^5 + 170291x^4 + 80926x^3 - 91913x^2 - 23405x - 227562
\\ \hline
7&11162x^5 + 249139x^4 - 81842x^3 - 137124x^2 + 213695x + 60867
\\ \hline
8&98634x^5 + 258878x^4 + 217977x^3 + 168018x^2 - 195678x - 57868
\\ \hline
9&37951x^5 - 49618x^4 + 171340x^3 + 166043x^2 + 81806x - 53074
\\ \hline
10&243858x^5 - 172816x^4 + 174721x^3 - 183549x^2 + 27505x + 1394
\\ \hline
11&196078x^5 + 7770x^4 + 173897x^3 - 240777x^2 - 62370x - 177733
\\ \hline
12&160271x^5 + 28136x^4 - 190262x^3 - 251564x^2 + 85945x + 96114
\\ \hline
13&-28204x^5 - 111716x^4 + 259330x^3 + 184365x^2 - 134038x - 188451
\\ \hline
14&157683x^5 + 180146x^4 - 175683x^3 + 207384x^2 + 60201x - 235015
\\ \hline
15&156049x^5 + 245627x^4 + 112291x^3 - 70153x^2 + 22757x - 112123
\\ \hline
16&-110436x^5 + 181478x^4 - 179573x^3 + 170745x^2 + 27505x + 178541
\\ \hline
17&196078x^5 + 7770x^4 + 173897x^3 - 240777x^2 - 62370x - 177733
\\ \hline
\end{array}
\end{eqnarray*}
\begin{eqnarray*}
\begin{array}{|l|c|}
\hline
n & j_n \\
\hline
1&-181949x^5 + 191925x^4 + 123820x^3 - 92832x^2 + 68256x - 32042
\\ \hline
2&70082x^5 - 126707x^4 + 223201x^3 - 162933x^2 - 241398x + 66137
\\ \hline
3&-72435x^5 - 250762x^4 + 80515x^3 + 174612x^2 - 75519x - 23327
\\ \hline
4&72391x^5 + 90594x^4 + 104200x^3 - 16026x^2 + 265050x - 231383
\\ \hline
5&45809x^5 + 89617x^4 - 56978x^3 - 260565x^2 - 41706x - 153832
\\ \hline
6&3543x^5 - 183856x^4 - 117449x^3 - 87666x^2 - 235494x - 68606
\\ \hline
7&-36959x^5 + 204318x^4 + 159118x^3 + 49341x^2 - 121563x - 205139
\\ \hline
8&173384x^5 - 180113x^4 - 221003x^3 - 71025x^2 - 172197x - 258079
\\ \hline
9&-97059x^5 - 63571x^4 + 153739x^3 + 126660x^2 + 162540x + 80110
\\ \hline
10&-155057x^5 + 263367x^4 + 8215x^3 - 95001x^2 + 167121x - 27992
\\ \hline
11&-43129x^5 + 95449x^4 + 133291x^3 - 149757x^2 - 93465x - 2200
\\ \hline
12&198186x^5 + 172625x^4 + 212788x^3 - 109536x^2 + 103491x + 198208
\\ \hline
13&-155057x^5 + 263367x^4 + 185362x^3 + 259293x^2 - 10026x - 27992
\\ \hline
14&-43129x^5 + 95449x^4 + 133291x^3 - 149757x^2 - 93465x - 2200
\\ \hline
15&198186x^5 + 172625x^4 + 212788x^3 - 109536x^2 + 103491x + 198208
\\ \hline
16&-155057x^5 + 263367x^4 + 185362x^3 + 259293x^2 - 10026x - 27992
\\ \hline
17&-43129x^5 + 95449x^4 + 133291x^3 - 149757x^2 - 93465x - 2200
\\ \hline
\end{array}
\end{eqnarray*}
Theorem \ref{teichmueller} and Corollary \ref{conv} imply that
$j_{11}=j_{17}$ which agrees with our computation.
One can see from the above computational evidence
that $a_{11}=a_{17}$ and $b_{11}=b_{17}$. 
The scalar which gives
the action of the absolute Frobenius lift $E_{11} \rightarrow E_{17}$
on differentials is congruent $153819$ modulo $3^{12}$.
By the procedure which is described
in Algorithm \ref{countpoints} one computes that the trace of Frobenius
of the reduction of $E$ equals $-38$.
\newline\indent
We have implemented the Algorithm \ref{countpoints}
in the computer algebra programming language Magma \cite{magma}.
Using our experimental implementation we were able to compute the
number of rational points of ordinary elliptic curves
over finite fields of characteristic $3$ of
cryptographic size in a reasonable amount of time.
For example, we computed the
number of points on the elliptic curve given by the equation
\begin{eqnarray}
\label{exx}
y^2 + xy = x^3 + s
\end{eqnarray}
where
{\small
\begin{eqnarray*}
s & = & 2z^{149} + 2z^{148} + 
    2z^{147} + z^{146} + z^{145} + 2z^{144} + z^{143} + 2z^{142} + 2z^{139} \\
&&    + z^{134} + 2z^{133} + z^{131} + z^{130} + 2z^{129} + z^{128} +
    z^{127} \\
&& + 2z^{126} + z^{123} + 2z^{122} + 2z^{120} + z^{119} + 2z^{118} +
    z^{117} \\
&&   + 2z^{113} + 2z^{111} + 2z^{110} + z^{109} + z^{106} + z^{105} +
    z^{102} \\
&&   + 2z^{101} + z^{99} + 2z^{97} + z^{95} + z^{94} + 2z^{93} + z^{91}
    \\
&& + 2z^{90} + z^{88} + 2z^{87} + 2z^{86} + z^{85} + 2z^{84} + 2z^{83}
    \\
&& + z^{82} + z^{81} + z^{80} + 2z^{79} + 2z^{78} + 2z^{76} + z^{75} +
    z^{74} \\
&& + 2z^{72} + 2z^{71} + z^{70} + z^{69} + 2z^{65} + z^{64} + z^{63} \\
&& + 2z^{62} + 2z^{59} + 2z^{58} + z^{57} + z^{56} + z^{55} + z^{54} \\
&& + 2z^{53} + z^{50} + z^{49} + z^{47} + z^{44} + z^{43} + 2z^{41} +
    z^{40} \\
&& + z^{38} + z^{37} + 2z^{36} + z^{33} + z^{32} + z^{30} + z^{28} +
    2z^{27} \\
&& + 2z^{24} + z^{23} + 2z^{19} + 2z^{13} + 2z^{12} + z^{11} + z^{10} \\
&& +   2z^8 + 2z^7 + z^5 + 2z^4 + z^3 + z
\end{eqnarray*}}
over $\xf_{3^{150}}=\xf_3[z]/(f)$
with $f(z)=z^{150}+z^5+z^4+z^3+z+2$.
We computed using Algorithm \ref{countpoints} that the curve (\ref{exx}) has
{\small
\[
369988485035126972924700782451696643480338123589346780021422648929803646
\]}
rational points over the finite field $\xf_{3^150}$.
The computation for this curve took $639$ seconds on an Intel Core 2 Duo
(E7400@2.80GHz) computer with 8GB memory.

\section{Perspectives}

We suggest another
potential application of the explicit formulae
of Section \ref{explicit}.
The example of the introduction shows that
also in the supersingular case the formulae
of Theorem \ref{froblift} can be
used to compute an explicit lift of the relative $p$-Frobenius
morphism. As one can see from the computational
evidence, the computations take place over a ring of integers
which is ramified at $p$.
It is an interesting question whether one
can use our formulae to compute
an explicit lift of the absolute
Frobenius morphism in the supersingular case.
We expect that the answer to this question is positive.
\newline\indent
The algorithms presented in Section \ref{liftalgo} and
Section \ref{pointcounting}
may be improved with respect to their complexity.
This improvement might be relevant in practice. 
The complexity bound given in Theorem \ref{comp} may be
improved to a bound that is linear in $p$.
In order to do so, one has to avoid computing with the $p$-th division
polynomial, which has degree of order $p^2$.
Perhaps it is possible to use the formal group law of the elliptic curve
$E$ instead. We remark that
the formal group incorporates the local part of the $p$-torsion,
or equivalently, the kernel of reduction.
Also it seems to be worthwhile trying to improve the complexity bound
of Theorem \ref{pccomp} with respect to $d$.
One expects that there is an algorithm whose complexity is
essentially quadratic in $d$.

\section{Acknowledgments}

I am grateful to Jaap Top for many
inspiring mathematical discussions which led to the
results of this article.
In fact, this article contains
the results of the first and second chapter of my PhD thesis
that he supervised.
I thank Ben Moonen for pointing out the very simple and
elegant proof of the convergence theorem that we give
in Section \ref{sectconv}. 

\bibliographystyle{plain}

\end{document}